[1994/12/01]
\documentclass[reqno,12pt]{amsart}
\title{On norm continuity, differentiability and compactness of perturbed semigroups}
\author{A. Boulouz, H. Bounit, A. Driouich and S. Hadd}
\address{Department of Mathematics, Faculty of Sciences, Ibn Zohr University, Hay Dakhla, BP8106, 80000--Agadir, Morocco; abed.boulouz@edu.uiz.ac.ma, h.bounit@uiz.ac.ma, a.driouich@uiz.ac.ma, s.hadd@uiz.ac.ma}
\thanks{}

\usepackage{amsmath,amsthm}
\usepackage{amsfonts}
\usepackage{amssymb}
\usepackage{latexsym}

\newtheorem{thm}{Theorem}[section]

\newtheorem{lem}[thm]{Lemma}

\theoremstyle{definition}

\newtheorem{defn}[thm]{Definition}
\newtheorem{exa}[thm]{Example}
\newtheorem{rem}[thm]{Remark}

\numberwithin{equation}{section}

\newcommand{\R}{{\mathbb R}}
\newcommand{\T}{{\mathbb T}}                   
\newcommand{\C}{{\mathbb C}}                  
\newcommand{\F}{{\mathbb F}}

\newcommand{\B}{{\mathbb B}}

\def\al{\alpha}
\def\om{\omega}

\def\ga{\gamma}
\def\si{\sigma}
\def\Si{\Sigma}
\def\la{\lambda}

\def\t{\tau}
\def\La{\Lambda}

\def\calM{{\mathcal{M}}}

\def\calL{{\mathcal{L}}}

\def\calT{{\mathcal{T}}}

\def\calG{{\mathcal{G}}}
\def\calA{{\mathcal{A}}}

\def\calX{{\mathcal{X}}}

\def\calZ{{\mathcal{Z}}}

\def\calP{{\mathcal{P}}}
\def\R{\mathbb R}

\def\P{\mathbb P}

\def\C{\mathbb C}

\def\F{\mathbb F}

\def\L{\mathbb L}

\def\T{\mathbb T}

\begin{document}
\maketitle

\renewcommand{\sectionmark}[1]{}
\begin{abstract}
The main purpose of this paper is to treat semigroups properties, like norm continuity, compactness and differentiability for perturbed semigroups in Banach spaces. In particular, we investigate three large classes of perturbations, Miyadera-Voigt,  Desch-Schappacher and Staffans-Weiss perturbations. Our approach is mainly based on feedback theory of Salamon-Weiss systems. Our results are applied to abstract boundary integro-differential equations in Banach spaces.
\end{abstract}

\section{Introduction}

In this paper, we investigate classical properties like norm continuity, compactness and differentiability for some classes of perturbed semigroups. To be more precise, let Banach spaces $X,Z$ such that $Z\subset X$,   $(A,D(A))$ a generator of a strongly continuous semigroup $\T:=(T(t))_{t\ge 0}$ on $X$ such that $D(A)\subset Z$ and a linear operator $\L\in\calL(\tilde{Z},\tilde{X})$ with $\tilde{Z}$ and $\tilde{X}$ are Banach spaces carefully chosen  in order that $A+\L$ with appropriate domain is well defined and  generates a strongly continuous semigroup $\T^{cl}:=(T^{cl}(t))_{\ge 0}$ on $X$ (this notation will be justified in the next section). Now the problem to be treated is: does semigroups  generated by  $A$ and $A+\L$ permute the aforementioned properties? As a matter of facts this problem is already considered by many authors and have some partial answers (depending on the type of perturbations).

The class of bounded perturbations, i.e. the case when $X=\tilde{Z}=\tilde{X}$ (so that $\L\in\calL(X)$) is mainly treated by  Phillips \cite{Phillips}. He  proved that if $(T(t))_{t\ge 0}$  is norm continuous (resp. compact) for $t>0$, then  the operator $(A+\L,D(A))$ generates a strongly continuous semigroup $(T^{cl}(t))_{t\ge 0}$ on $X$ is norm continuous (resp. compact) for $t>0$, as well. On the other hand, Phillips constructed a semigroup $(T(t))_{t\ge 0}$ which is norm continuous for $t>t_0$ with $t_0>0$ (i.e. eventually norm continuous)  but the semigroup $(T^{cl}(t))_{t\ge 0}$ is not norm continuous for $t>t_0$. Thus  eventual norm continuity and eventual compactness are, in general, not preserved under even bounded perturbations. It is shown in \cite{Engel Nagel} that in the case of compact perturbation operator $\L\in\calL(X)$, the eventual norm continuity is preserved for the perturbed semigroup whenever the initial semigroup is. In 1983, Pazy \cite{Pazy} (see also \cite{Renardy}) showed that the eventual differentiability of the semigroup $(T(t))_{t\ge 0}$ is not  translated to the perturbed semigroup generated by  $A+\L$ even if $\L$ is a bounded perturbation.  As shown in \cite{Batty1},   \cite{Batty}, \cite{Doytchinov}, \cite{Iley}, and \cite{Zhang.L}, extra conditions on the semigroup $\T$ are needed to assure the preservation of the eventual differentiability for the semigroup generated by $A+\L$.

Let us now analysing in profile the case of unbounded perturbations.  Three large classes of unbounded perturbations will be investigated. The first class: choose $\tilde{Z}=Z$ and $\tilde{X}=X$, then we say that $\L\in\calL(Z,X)$ is a {\em Miyadera-Voigt perturbation} for $A$ if
 there exist $\al>0$ and $\ga\in (0,1)$ such that for any $x\in D(A)$ we have
\begin{align*}
\int^\al_0 \|\L T(t)x\|dt\le \ga \|x\|.
\end{align*}
In this case the operator $(A+\L,D(A))$ generates  a strongly continuous semigroup $(\T^{cl}(t))_{t\ge 0}$ on $X$. As an example of application we cite the case of delay equations \cite{batkai}, \cite{Hadd-SF}. In general, we do not  have preservation of the aforementioned regularities under  Miyadera-Voigt perturbations. However,  for delay evolution equations with $L^p$--history spaces it show in \cite{batkai} (resp. \cite{Matrai1}) that if the free delay equation is governed by an immediately norm continuous semigroup (resp. immediately compact), then the delay semigroup in product spaces  associated with the delay equation is eventually norm continuous (resp. eventually compact). Moreover, some results on eventual differentiability for delay equations is obtained in \cite{Batty}. The second class of unbounded perturbations is the case $\tilde{Z}=X$ and $\tilde{X}=X_{-1},$ so that $\L\in\calL(X,X_{-1}),$ where $X_{-1}$ is the completion of $X$ with respect to the norm $\|x\|_{-1}:=\|(\la I-A)^{-1}x\|,\;x\in X$. The semigroup $\T$ can be extended to a strongly continuous semigroup $\T_{-1}:=(T_{-1}(t))_{t\ge 0}$ on $X_{-1},$ whose generator $A_{-1}:X\to X_{-1},$ the extension of $A$ to $X$. In this case we say that $\L$ is a {\em Desch-Schappacher perturbation} for $A$ if there exists $t_0>0$ such that
\begin{align*}
\int^{t_0}_0 T_{-1}(t_0-s)\L f(s)ds\in X,\qquad \forall f\in L^1([0,t_0],X).
\end{align*}
It is well-know (see \cite[chapter III-3-d]{Engel Nagel}) that for a such $\L,$ the part of the operator $A_{-1}+\L$ on $X$ (denoted by $A^{cl}:=(A_{-1}+\L)_{|X}$) generates a strongly continuous semigroup $\T^{cl}$ on $X$. For this kind of perturbations, M\'{a}trai \cite{Matrai} has showed that $\T^{cl}$ is immediately norm-continuous whenever the semigroup  $\T$ is (see also Jung \cite{Jung}). To the best of our knowledge, there is no results concerning the differentiability under Desch-Schappacher perturbations. This is one of the objectives  of this paper. In fact, we will show that if the generator $A$ satisfies  the so-called the Pazy condition \cite[p.57]{Pazy} (see also \eqref{Pazy condition} below), then the semigroup generated by $A^{cl}$ is differentiable. Finally, let us discuss another more general class of perturbations. To that purpose, let $A_m:Z\subset X\to X$ be a differential linear closed operator and $G,M:Z\to U,$ boundary linear operators, where $U$ is a (boundary) Banach space. We consider the  linear operator on $X,$
\begin{align*}\mathcal{A}:=A_m,\quad \quad D(\mathcal{A})=\{f\in Z,\quad  Gf=Mf\}.\end{align*}
We assume that $A:=A_m$ with domain $D(A)=\ker(G)$ is a generator of a strongly continuous semigroup $\T$ on $X$. Observe that the operator $\calA$ is obtained by perturbing the domain $D(A)$ of $A$ by an unbounded perturbation $M$. Based on feedback theory of regular linear systems (\cite{Weiss}, see also the next section for definitions), the authors of
\cite{Hadd Manzo Ghandi} introduced sufficient conditions on $G$ and $M$ for which  $(\calA,D(\calA))$ is a generator of a strongly continuous semigroup $\T^{cl}$ on $X$. In fact, they proved that there exists a space $\tilde{Z}$ such that $D(A)\subsetneq Z\subset \tilde{Z}\subsetneq X$ and an operator $\L\in\calL(\tilde{Z},X_{-1})$ such that the operator $(\calA,D(\calA))$ coincides with the following one
\begin{align*}
A^{cl}:=A_{-1}+\L,\qquad D(A^{cl})=\left\{x\in \tilde{Z}: (A_{-1}+\L)x\in X\right\}.
\end{align*}
In this case, the operator $\L\in\calL(\tilde{Z},X_{-1})$ is called a {\em Staffans--Weiss perturbation} (see Theorem \ref{staffans-weiss perturbation}).
 The main objective of this work is to prove that immediate norm continuity and compactness are preserved under Staffans--Weiss perturbations operators, see Theorem \ref{main thm I N C} and Theorem \ref{main thm I C} below. Particular case of these results is when the operator $M$ is bounded, i.e. $M\in\calL(X),$ so that we are in  Desch-Schappacher perturbations setting.

As source of applications of our abstract results, we will consider regularity of solutions of the following intergo-differential equation
\begin{equation*}
  \left\{
    \begin{array}{ll}
      \dot{x}(t)=A_{m}x(t)+\displaystyle\int_{0}^{t}k(t-s)Px(s)ds, & \hbox{$t\geq 0$} \\
       Gx(t)=Mx(t)  , & \hbox{$t\geq 0$,}\\
      x(0)=x, & \hbox{}
    \end{array}
  \right.
\end{equation*}
where $A_m,G$ and $M$ as above, $P:Z\to X$ is an admissible observation operator for $A,$ and the kernel $k(\cdot)$ belong to an appropriate Bergman space (see Section \ref{Appl-Volterra}).

For the reader's convenience, we briefly recall the relevant background from \cite{Weiss} (and also \cite{Hadd Manzo Ghandi})
and related works and introduce (much of) our notation in Section \ref{Salamon-Weiss-systems}. Section \ref{INC-Staffans-Weiss} is on the study of immediate norm continuity and compactness of semigroups under Staffans-Weiss perturbations operators. In section \ref{eventual-section} we investigate
the eventual differentiability of semigroups under Desch-Schappacher perturbations. The last section is concerned with the study of a class of
integro-differential equations in Banach spaces.

\section{Staffans-Weiss perturbation theorem}\label{Salamon-Weiss-systems}
In this section, we shall recall the recent concept of Staffans-Weiss perturbations. The origin of these  perturbations is the feedback theory of well-posed and regular linear systems introduced mainly by Salamon, Staffans and Weiss, see e.g. \cite{Salamon}, \cite[chap.7]{Staffans} and \cite{Weiss}.

Throughout this section $X$ and $U$ are Banach spaces (with norms denoted (for simplicity) by the same symbol $\|\cdot\|$) and $p>1$ is a real number. Let $Z$ be another Banach space such that $Z\subset X$ (with continuous and dense embedding). We now consider a differential operator $A_m:Z\to X$ and a trace operator $G:Z\to U$ assumed to be surjective. We also assume that the following operator
\begin{align}\label{generator-A}
A=A_m\quad D(A)=\left\{x\in Z: Gx=0\right\}
\end{align}
generates a strongly continuous semigroup $\mathbb{T}:=(T(t))_{t\ge 0}$ on $X$ of type $\om_0(A)$. We denote by $\rho(A)$ the resolvent set of $A,$ $\si(A)=\mathbb{C}\backslash\rho(A)$ the spectrum of $A$, and $R(\la,A):=(\la I-A)^{-1}$ for $\la\in\rho(A),$ the resolvent operator of $A$. The graph norm with respect to $A$ is $\|x\|_A:=\|x\|+\|Ax\|$ for $x\in D(A)$. It is well-known that $X_A:=(D(A),\|\cdot\|_A)$ is a Banach space and $X_A\hookrightarrow X$ (densely and continuously). It is shown by Greiner \cite{Greiner} that the restriction of $G$ to $\ker(\la-A_m)$ for $\la\in\rho(A)$ is invertible. As known, the inverse
\begin{align}\label{Dirichlet-operator}
D_\la:=\left(G_{|\ker(\la-A_m)}\right)^{-1}\in\calL(U,Z)
\end{align}
is called the Dirichlet operator. On the other hand, we define a new norm on $X$ by setting $\|x\|_{-1}:=\|R(\beta,A)x\|$ for $x\in X$ and some (hence all) $\beta\in \rho(A)$. The completion of $X$ with respect to this norm is a Banach space denoted by $X_{-1}$ and satisfies
\begin{align*}
X_A\hookrightarrow X\hookrightarrow X_{-1}.
\end{align*}
The extension of the semigroup $\T$ to $X_{-1}$ is again a strongly continuous semigroup $\mathbb{T}_{-1}:=(T_{-1}(t))_{t\ge 0}$ on $X_{-1}$ whose generator $A_{-1}:D(A_{-1})=X\to X_{-1}$ is the extension of $A$ to $X$. We consider the boundary control problem
\begin{align}\label{BCP}
\begin{cases} \dot{z}(t)=A_m z(t),& t\ge 0,\cr z(0)=x,\cr Gz(t)=u(t),& t\ge 0,
\end{cases}
\end{align}
where $u:[0,+\infty)\to U$ is a control function (boundary control). If one looks for a solution $z\in C([0,\al],X)$ (with $\al>0$) of \eqref{BCP}, it is more convenient to reformulate a such boundary problem as a distributed one. To that purpose, we  introduce the following operator
\begin{align}\label{operator B}
B:=(\la-A_{-1})D_\la\in\calL(U,X_{-1}),\qquad \la\in\rho(A),
\end{align}
(this operator do not depend of $\la$, due to the resolvent equation). One can see that the boundary problem \eqref{BCP} is equivalent to
\begin{align}\label{DCP}
\begin{cases} \dot{z}(t)=A_{-1} z(t)+B u(t),& t\ge 0,\cr z(0)=x.
\end{cases}
\end{align}
The solution of \eqref{DCP} (and hence of \eqref{BCP}) is given by
\begin{align}\label{adm cont}
z(t)&=T(t)x+\int_{0}^{t}T_{-1}(t-s)Bu(s)ds\cr &:=T(t)x+\Phi_{t}u,
\end{align}
for any $t\ge 0,\;x\in X$ and $u\in L^p([0,+\infty),U)$. We have
\begin{align*}
\Phi_t\in\calL(L^p([0,t],U),X_{-1}),\qquad \forall t>0.
\end{align*}
Observe that the solution $z$ takes value in $X_{-1}$. We then have the following definition.
\begin{defn}
  The operator $B\in \mathcal{L}(U,X_{-1})$ is called an admissible control operator for $A$ if there exists $t_{0}>0$ such that ${\rm Range}(\Phi_{t_0})\subset X$. In this case we also say that the pair $(A,B)$ is admissible or sometimes well-posed.
 \end{defn}
If $(A,B)$ is well-posed then we have $\Phi_t\in\calL(L^p([0,t],U),X)$ for any  $t>0$ and the solution of \eqref{BCP} satisfies $z\in C(\R^+,X)$, see \cite{Weiss2}, \cite[chap.4]{Tucsnak Weiss}. The familly $(\Phi_t)_{t\ge 0},$ satisfies for all $t,\t\ge 0,$
\begin{align}\label{Phi-EF}
\Phi_{t+\t}u=T(t)\Phi_\t(u_{|[0,\t]})+\Phi_t u(\cdot+\t)
\end{align}
for any $u\in L^p([0,t+\t],U)$. In addition
\begin{align}\label{phi-comparaison}
\|\Phi_{\t_1}\|\le \|\Phi_{\t_2}\|
\end{align}
for any $0\le \t_1\le \t_2$. Moreover,  for all $\omega > \omega_{0}(A)$, there exists a constant $c>0$ such that
\begin{equation}\label{adm cont weiss}
\|R(\lambda,A_{-1})B\|\leq \frac{c}{({\rm Re} \lambda - \omega)^{\frac{1}{p}}}
\end{equation}
for any $\la\in\C$ such that ${\rm Re} \lambda>\om$.

Let us now consider the following observed boundary system
\begin{align}\label{BOP}
\begin{cases} \dot{z}(t)=A_m z(t),& t\ge 0,\cr z(0)=x,\cr Gz(t)=0,& t\ge 0,\cr y(t)=M z(t),& t\ge 0.
\end{cases}
\end{align}
where $M:D(M)=Z\to U$ is a linear operator and $y:[0,+\infty)\to U$ is an observation function. The problem is how to extend the observation function $y(\cdot;x)$ to a function in $L^p_{loc}(\R^+,U)$ for any initial condition $x\in X$. To that purpose, let $(A,D(A))$ be defined by \eqref{generator-A} and select
\begin{align}\label{Operator C}
C=M\imath\in\calL(X_A,U),
\end{align}
where $\imath:D(A)\to Z$ is the continuous injection. Then the system \eqref{BOP} is reformulated as
\begin{align}\label{DOP}
\begin{cases} \dot{z}(t)=A z(t),& t\ge 0,\cr z(0)=x,\cr y(t)=C z(t),& t\ge 0.
\end{cases}
\end{align}
The mild solution of the evolution equation in \eqref{DOP} is given by $z(t)=T(t)x$ for any $t\ge 0$ and $x\in X$. However the observation $y(t)$ is only defined on the domain $D(A)$ and we have $y(t)=CT(t)x$ for any $t\ge 0$ and initial condition $x\in D(A)$. We need the following definition
\begin{defn}
  $C\in\mathcal{L}(D(A),U)$ is called an admissible observation operator for $A$ (we also say that $(C,A)$ is admissible or well-posed) if
\begin{equation}\label{adm obs}
 \int_{0}^{\tau}\|CT(s)x\|^{p}ds\leq \gamma^{p}(\t)\|x\|^{p}
\end{equation}
for all $x\in D(A)$ and for some constants $\tau>0$ and $\gamma:=\gamma(\tau)>0$.
\end{defn}
To state our main results in next sections, we need to define the concept of zero class admissible observation operators which was first introduced in \cite{XLY}, in order to provide conditions for exact observability of semigroup systems. This concept was further developed in \cite{Jacob Partington Pott}.
\begin{defn}\label{zero-class-admissible}
The operator $C\in\calL(D(A),U)$ is said to belong to the zero class of admissible observation operators
for $A$ ($C$ is zero-class admissible), if the best constant $\ga(\t)$,
given by \eqref{adm obs}, satisfies $\ga(\t)\to 0$ as $\t\to 0$.
\end{defn}
Obviously, bounded observation operators $C\in \calL (X,U)$ are zero-class admissible.\\
Let $(C,A)$ be admissible. Due to \eqref{adm obs} and the density of the domain $D(A)$ in $X,$ the linear operator
\begin{align*}
\Psi_\infty: D(A)\to L^p_{loc}([0,+\infty),U),\quad x\mapsto \Psi_\infty x=CT(\cdot)x
\end{align*}
is extended to a linear bounded operator $\Psi_\infty: X\to L^p_{loc}([0,+\infty),U)$. Then the observation function $y$ can be extended to a $p$-locally integrable function for any initial condition $x\in X$ by setting
\begin{align*}
y(t)=(\Psi_\infty x)(t),\qquad a.e.\;t\ge 0
\end{align*} Next we recall the  representation of $\Psi_\infty$ (then $y$) using an extension operator of $C$. We then need the following concept.
\begin{defn}
  The Yosida extension of an operator $C\in\mathcal{L}(D(A),U)$ with respect to $A$ is the operator defined by
\begin{align*}
  C_{\Lambda}x:= & \lim_{\lambda \rightarrow +\infty}C\lambda R(\lambda,A)x \\
  D(C_{\Lambda}):= & \{x\in X,\lim_{\lambda \rightarrow +\infty}C\lambda R(\lambda,A)x \mbox{ exists in } U \}.
\end{align*}
\end{defn}
Let $\lambda_{0}$ such that $[\lambda_{0},\infty)\subset\rho(A)$. We define the following norm
$$
\|x\|_{D(C_{\Lambda})}:=\|x\|_{X}+\sup_{\lambda\geq \lambda_{0}}\|C\lambda R(\lambda,A)x\|_{U},\qquad x\in D(C_\Lambda).
$$
According to \cite{Weiss3}, $[D(C_{\Lambda})]:=(D(C_{\Lambda}),\|\cdot\|_{D(C_{\Lambda})})$ is a Banach space. Now if  $(C,A)$ is admissible, the representation theorem of Weiss (see \cite{Weiss3}) shows that
  ${\rm Range}(T(t))\subset D(C_{\Lambda})$ for a.e. $t>0$. On the other hand, for any $x\in X$ and a.e. $t\ge 0,$ we have
  \begin{align*}
  \left(\Psi_\infty x\right)(t)=C_\Lambda T(t)x.
  \end{align*}
A necessary condition for the well-posedness of $(C,A)$ is: for all $\omega\in \C $ with $\omega > \omega_{0}(A)$, there exists a constant $b>0$ such that
\begin{equation}\label{adm obs weiss}
\|CR(\lambda,A)\|\leq \frac{b}{({\rm Re} \lambda - \omega)^{1-\frac{1}{p}}},\qquad {\rm Re} \lambda>\om,
\end{equation}
see \cite{Staffans}, \cite{Tucsnak Weiss}, and \cite{Weiss3}, for more details.

Let us now consider the input-output boundary system
\begin{equation}\label{input output system}
  \left\{
  \begin{array}{ll}
    \dot{z}(t)=A_m z(t), & \hbox{$t\geq 0$,} \\
    Gz(t)=u(t), & \hbox{$t\geq 0$,} \\
    y(t)=Mz(t), & \hbox{$t\geq 0$,} \\
    z(0)=x, & \hbox{}
  \end{array}
\right.
\end{equation}
The question now is how to define the well-posedness of the system \eqref{input output system} in the sense that the solution $z(t;x,u)$ should belong to the state space $X$ for any $x\in X,\;u\in L^p([0,+\infty),U)$ and a.e. $t\ge 0,$ continuously depends on $t$ and $u$, and the transformation $u\mapsto y$ define a linear bounded operator on $L^p_{loc}([0,+\infty),U)$. Using operators \eqref{operator B} and \eqref{Operator C}, the system \eqref{input output system} is reformulated as
\begin{align}\label{system ABC}
\left\{
  \begin{array}{ll}
    \dot{z}(t)=Az(t)+Bu(t), & \hbox{$t\geq 0$,} \\
    z(0)=x, & \hbox{} \\
    y(t)=Cz(t), & \hbox{$t\geq 0$.}
  \end{array}
\right.
\end{align}
We start by assuming that  $(A,B)$ and $(C,A)$ are admissible. In this case,  $z(t)=T(t)x+\Phi_t u\in X$ for any $t\ge 0$ and $u\in L^p([0,+\infty),U)$. We now look for additional condition so as to extend the output function $y$ to a function in $L^p_{loc}(\R^+,U)$. For this, define the space
\begin{align*}
W^{2,p}_{0,loc}([0,+\infty),U):=\left\{v\in W^{2,p}_{loc}([0,+\infty),U):v(0)=v'(0)=0\right\},
\end{align*}
which is dense in $L^p_{loc}([0,+\infty),U)$. Without loss generality, we can assume that $0\in \rho(A)$, so $B=(-A_{-1})D_0$ (see \eqref{operator B}). Now an integration by parts yields, for any $u\in W^{2,p}_{0,loc}([0,+\infty),U),$
\begin{align*}
\Phi_t u=D_0 u(t)-\int^t_0 T(t-s)D_0 u'(s)ds\in Z,\qquad t\ge 0.
\end{align*}
This allows us to define the following operator
\begin{align*}
\left(\F_\infty u\right)(t)=M\Phi_t u,\qquad u\in W^{2,p}_{0,loc}([0,+\infty),U),\;a.e.\;t\ge 0.
\end{align*}

\begin{defn}
  The system $(A,B,C)$ is called  well-posed on $X,U,U$ if
  \begin{enumerate}
    \item $(A,B)$ is well-posed on $X,U,$
    \item $(C,A)$ is well-posed on $X,U$, and
    \item there exist constants $\t>0$ and $\kappa_\t >0$ such that
    \begin{align}\label{F-infty-estimate}
    \left\|\F_\infty u\right\|_{L^p([0,\t],U)}\le \kappa_\t\;\left\|u\right\|_{L^p([0,\t],U)},\qquad u\in W^{2,p}_{0,loc}([0,+\infty),U).
    \end{align}
  \end{enumerate}
  \end{defn}
  The first consequence of the well-posedness of the system $(A,B,C)$ is that the operator $\F_\infty$ can be extended to a linear bounded operator on $L^p_{loc}([0,+\infty),U)$, denoted again by $\F_\infty$. On the other hand, observation function of the system \eqref{input output system} is given by
  \begin{align*}
  y=\Psi_{\infty}x+\F_\infty u,\qquad (x,u)\in X\times L^p_{loc}([0,+\infty),U),
\end{align*}
which is a function in $L^p_{loc}([0,+\infty),U)$. Observe that the Laplace transform of $\F_\infty$ is given by
\begin{align}\label{Laplace-F-infty}
\widehat{(\F_\infty u)}(\la)=MR(\la,A_{-1})B \hat{u}(\la)
\end{align}
for $\la\in\rho(A)$ such that these Laplace transforms exist.

Now in order to given a representation of the observation function $y(t)$ is terms of the operator $C$ and the state $z(t)$ of the system \eqref{input output system}, we need the following important subclass of well-posed systems.
\begin{defn}
A well-posed system $(A,B,C)$ is  regular on $X,U,U$ (with feedthrough $D=0$) if
$$
\lim_{t\rightarrow 0^{+}}\frac{1}{t}\int_{0}^{t}(\F_{\infty}u_{0})(\sigma)d\sigma =0
$$
exists in $U$, for the constant control function $u_{0}(t)=v$, $v\in U$, $t\geq 0$.
\end{defn}
If the system $(A,B,C)$ is regular and by using a tauberian theorem we have
\begin{align*}
\lim_{\la>0,\;\la\to +\infty} \la \widehat{(\F_\infty u_0)}(\la)=0.
\end{align*}
According to \eqref{Laplace-F-infty}, if $(A,B,C)$ is regular then
\begin{align}\label{limit-Transfer-function}
\lim_{\la\to+\infty} M R(\la,A_{-1})B v=0,\qquad v\in U.
\end{align}
In addition, if we take $\la>0$ sufficiently large and $\mu\in\rho(A)$, then by using \eqref{limit-Transfer-function}, and the following identity
\begin{align*}
C \la R(\la,A)R(\mu,A_{-1})B v=\frac{\la}{\la-\mu}MR(\mu,A_{-1})Bv-\frac{1}{\la-\mu}M\la R(\la,A_{-1})Bv,\qquad v\in U,
\end{align*}
we have ${\rm Range}(R(\mu,A_{-1})B)\subset D(C_\La)$ and
\begin{align*}
C_\Lambda R(\mu,A_{-1})B v=M R(\mu,A_{-1})B v,\qquad v\in U,\; \la\in\rho(A).
\end{align*}
With these observations and using \cite[lem.3.6]{Hadd Manzo Ghandi}, we have
\begin{align}\label{inclusion-important}
Z\subset D(C_\Lambda)\quad\text{and}\quad (C_\Lambda)_{|Z}=M.
\end{align}
We also mention that (see \cite{Weiss1}) if $(A,B,C)$ is regular then ${\rm Range}(\Phi_t)\subset D(C_\Lambda)$ (and in particular the state of \eqref{input output system} $z(t)\in D(C_\Lambda)$) for a.e. $t\ge 0$. Moreover, $(\F_\infty u)(t)=C_\Lambda \Phi_t u$ for all $u\in L^p_{loc}([0,+\infty),U)$ and a.e. $t\ge 0$. In addition, the observation function of the system \eqref{input output system} is given by
\begin{align*}
y(t)=C_\Lambda z(t),\qquad a.e.\;t>0
\end{align*}
for any initial condition $x\in X$ and any control function $u\in L^p_{loc}([0,+\infty),U)$, see \cite{Weiss1} for more details.

In many cases, it is very important to work not only with the linear operator  $\F_\infty$ but also with its Laplace transform, called the transfer function. So if a system $(A,B,C)$ is regular then its transfer function if given by
\begin{center}
  $H(\lambda):=C_{\Lambda}R(\lambda,A_{-1})B$, $\quad Re \lambda > \omega_{0}(A)$.
\end{center}
According to Weiss \cite{Weiss1}, there exists $\gamma >0$ such that
\begin{align}\label{estimation-transfer-function}
\sup_{Re \lambda > \gamma}\|H(\lambda)\|<+\infty.
\end{align}
In the rest of this section we shall present a perturbation theorem associated with regular linear systems, due to Weiss \cite{Weiss1} in the Hilbert setting and Staffans \cite[chap.7]{Staffans} for the Banach cases. To that purpose we need the following definition.
\begin{defn}
Let $(A,B,C)$ be a regular linear system on $X,U,U$. The identity operator $I_U:U\to U$ is called an admissible feedback  if $I-\F_{\infty}$  has uniformly bounded inverse.
\end{defn}

\begin{thm}\label{staffans-weiss perturbation}
Assume that the system $(A,B,C)$ is regular on $X,U,U$ and $I_U$ is an admissible feedback. Let $C_\Lambda$ be the Yosida extension of $C$ with respect to $A$. Then the operator
\begin{align}\label{Operator-closed}
\begin{split}
  A^{cl}:= & A_{-1}+BC_{\Lambda} \\
  D(A^{cl}):= & \{x\in D(C_{\Lambda}): (A_{-1}+BC_{\Lambda})x\in X\}
  \end{split}
\end{align}
generates a $C_{0}$-semigroup $\T^{cl}:=(T^{cl}(t))_{t\geq0}$ on $X$ such that
\begin{itemize}
  \item [{\rm(i)}] ${\rm Range}(T^{cl}(t))\subset D(C_\Lambda)$ for a.e. $t>0$,
  \item [{\rm(ii)}] for any $x\in X$ and $t\ge 0,$
  \begin{align}\label{VCF-T-cl}
 T^{cl}(t)x=T(t)x+\int_{0}^{t}T_{-1}(t-s)BC_{\Lambda}T^{cl}(s)xds,
\end{align}
  \item [{\rm(iii)}] there exist constants $\al>0$ and $\delta_\al>0$ such that
  \begin{align*}
  \int^\al_0 \left\|C_\Lambda T^{cl}(t)x\right\|^p\le \delta_\al^p \|x\|^p
  \end{align*}
  for any $x\in X$.
\end{itemize}
\end{thm}

Let us denote by $X^A_{-1} $ (resp. $X^{A^{cl}}_{-1}$) be the extrapolation space associated with $X$ and $A$ (resp. $X$ and $A^{cl}$). Obviously these spaces are  different. In \cite{Weiss}, Weiss constructed subspaces $W_A$ and $W_{A^{cl}}$ of $X^A_{-1} $ and $X^{A^{cl}}_{-1}$, respectively, such that $Jx:=\displaystyle\lim_{\lambda\to\infty}\lambda R(\lambda,A_{-1})x$ in $X^{A^{cl}}_{-1}$ defines an isomorphism $J: W_A\longrightarrow W_{A^{cl}}$. Obviously, $Jx=x$ in $X$. We obtain
\begin{align}\label{isomorphism J}
\int_{0}^{t}T_{-1}(t-s)BC_{\Lambda}T^{cl}(s)xds=\int_{0}^{t}T^{cl}_{-1}(t-s)JBC_{\Lambda}T(s)xds,
\end{align}
see \cite[p:54-55]{Weiss} and \cite[Remark 4.6(b)]{schnaubelt} for more detail. Whence, from \eqref{isomorphism J}, the perturbed semigroup $(T^{cl}(t))_{t\geq 0}$  satisfies also the following variation of constants formula
\begin{equation}\label{convolution R}
 T^{cl}(t)x=T(t)x+\int_{0}^{t}T^{cl}_{-1}(t-s)JBC_{\Lambda}T(s)xds
\end{equation}
for any $t\ge 0$ and $x\in X$.

\begin{rem}\label{DS-MV-perturbations}
Assume that the operator $M$ is linear bounded from $X$ to $U$. In this case, we take $C=M$ is an admissible operator for $A$ and $C_\Lambda=C=M$. In this case, the operator $(A^{cl},D(A^{cl}))$ is exactly the part of the operator $A_{-1}+BM$ in $X$. Moreover if $(A,B)$ is well-posed, then the system $(A,B,M)$ is regular with $I_U$ is an admissible feedback operator. By applying Theorem \ref{staffans-weiss perturbation}, the operator $(A^{cl},D(A^{cl}))$ generates a $C_0$-semigroup $\T^{cl}$ on $X$ such that
 \begin{align}\label{VCF-T-cl-desch}
 T^{cl}(t)x=T(t)x+\int_{0}^{t}T_{-1}(t-s)BMT^{cl}(s)xds,
\end{align}
for any $t\ge 0$ and $x\in X$.
\end{rem}

\begin{rem}\label{unbounded perturbation domain}
Consider the boundary value problem
\begin{align}\label{perturbed boundary systems}
 \begin{cases}
    \dot{z}(t)=A_m z(t), & t\ge 0, \cr
    Gz(t)=Mz(t), & t\ge 0, \cr
    z(0)=x,
  \end{cases}
\end{align}
where $A_m:Z\to X$ and $G,M:Z\to U$ as defined at the beginning of this section. The problem \eqref{perturbed boundary systems} can be viewed as a partial differential equation where the boundary operator $G$ is perturbed by another unbounded trace operator $M$. This system can also  be reformulated as the following Cauchy problem in $X,$
\begin{align}\label{CP-equiv}
\begin{cases}
    \dot{z}(t)=\mathcal{A}z(t), & t\ge 0, \cr
    z(0)=x,
  \end{cases}
\end{align}
where
\begin{align}\label{Hadd-calA}
\calA:=A_m, \quad D(\mathcal{A})=\left\{x\in Z,\quad  Gf=Mf\right\}.
\end{align}
Then the  problem \eqref{perturbed boundary systems} is well-posed if the operator $(\calA,D(\calA))$ is a generator of a $C_0$-semigroup on $X$. Let the operator $A,B$ and $C$ defined by \eqref{generator-A}, \eqref{operator B} and \eqref{Operator C}, respectively. It is shown in \cite{Hadd Manzo Ghandi} that if $(A,B,C)$ is regular and $I_U$ is admissible feedback then $\calA$ coincides with the operator $A^{cl}$ defined by \eqref{Operator-closed}. Now according to Theorem \ref{staffans-weiss perturbation}, the operator $\calA$ generates a $C_0$-semigroup $\calT:=(\calT(t))_{t\ge 0}$ on $X$ given by
\begin{align*}
 \calT(t)x=T(t)x+\int_{0}^{t}T_{-1}(t-s)BC_{\Lambda}\calT(s)xds,
\end{align*}
for all $t\ge 0$ and $x\in X$. On the other hand it is shown in \cite{Hadd Manzo Ghandi} that for any $\lambda\in\rho(A)$ we have
$$
\lambda\in\rho(\calA)\Longleftrightarrow 1\in\rho(D_{\lambda}M) \Longleftrightarrow 1\in\rho(MD_{\lambda}).
$$
Moreover, for $\lambda\in\rho(A)\cap\rho(\mathcal{A})$,
\begin{align}\label{resolvent}
R(\lambda,\mathcal{A})=(I-D_{\lambda}M)^{-1}R(\lambda,A).
\end{align}
\end{rem}

\section{Immediate norm continuity and compactness of perturbed semigroups}\label{INC-Staffans-Weiss}
In this section, we will work under assumptions of Theorem \ref{staffans-weiss perturbation} (and also of Remark \ref{unbounded perturbation domain}). We then suppose that semigroup $\T$ is norm continuity or compact and then show if the perturbed semigroup $\T^{cl}$ inherits such properties.

Let us first  introduce a short proof of a result proved in \cite{Matrai} for Desch-Schappacher perturbations (i.e. in the case when the observation operator $C\in\calL(X,U)$ in Theorem \ref{staffans-weiss perturbation} or the boundary operator $M\in\calL(X,U)$ in Remark \ref{unbounded perturbation domain}). Before doing so, we recall from  \cite{Goersmeyer Weis} the following characterization of immediately norm continuity for strongly continuous semigroups in Banach spaces.

\begin{thm}\label{L.Weis}
  A strongly continuous semigroup $(T(t))_{t\geq 0}$ on a Banach $X$ is continuous in the operator norm for $t>0$ if and only if for all $\tau > 0$ the operator
  \begin{align*}
  K:L^r([0,\tau],X)\longrightarrow L^r([0,\tau],X),\qquad \left(Kf\right)(t)=\int^{t}_{0}T(t-s)f(s)ds
  \end{align*}
   satisfies the following Riesz condition $(R_{r})$ for some (all) $r\in (1,\infty)$ i.e :
  \begin{align*}
 (R_{r})\quad \int_{0}^{\tau}\|(Kf)(t+h)-(Kf)(t)\|^{r}dt\rightarrow 0 \text{ as } h\rightarrow 0
  \end{align*}
  uniformly for $f\in L^{r}([0,\tau],X)$ with $\|f\|_r\leq 1$.
\end{thm}
We are now in the position to give a new proof  \cite[Theorem 6]{Matrai}.
\begin{thm}
 Let the control system $(A,B)$ be well-posed on $X,U$ and $C\in \mathcal{L}(X,U)$. Assume that the semigroup $\mathbb{T}=(T(t))_{t\geq 0}$ is immediately norm continuous on $X$. Then the semigroup $\mathbb{T}^{cl}=(T^{cl}(t))_{t\geq 0}$ is immediately norm continuous on $X$.
\end{thm}
\begin{proof}
Let $0<h<\tau$ and $f\in L^{p}([0,\tau],X)$  with $\|f\|_p\leq 1$. We put
\begin{align*}
(K^{cl}f)(t):=\int_{0}^{t}T^{cl}(t-s)f(s)ds
\end{align*}
We will prove that if $K$ satisfies ($R_p$), then $K^{cl}$ also verifies ($R_p$). In fact, By using \eqref{convolution R}, a change of variables and Fubini theorem,   we obtain
\begin{align}\label{EEEEE}
(K^{cl}f)(t)=(Kf)(t)+\Phi^{J}_{t}C Kf,
\end{align}
where we set
\begin{align*}
\Phi^{J}_{t}v(\cdot)=\int_{0}^{t}T^{cl}_{-1}(t-s)JBv(s)ds,\qquad t\geq 0, v\in L^{p}([0,t],U).
\end{align*}
In view of \eqref{Phi-EF} and \eqref{EEEEE},
\begin{align*}
  (K^{cl}f)(t+h)-(K^{cl}f)(t) & =(Kf)(t+h)-(Kf)(t)+\Phi^{J}_{t+h}C Kf-\Phi^{J}_{t}CKf \\
   & =(Kf)(t+h)-(Kf)(t)+T^{cl}(t)\Phi^{J}_{h}CKf\\
   &\hspace{2.5cm}+\Phi^{J}_{t}C[(Kf)(\cdot+h)-Kf].
\end{align*}
By admissibility of $JB$ for $\mathbb{T}^{cl}$, we obtain
\begin{align*}
  \int_{0}^{\tau}\|(K^{cl}f)(t+h)-(K^{cl}f)(t)\|^{p}dt & \leq c_{p}\int_{0}^{\tau}\|(Kf)(t+h)-(Kf)(t)\|^{p}dt\\
   &+c_{p}\int_{0}^{\tau}\|T^{cl}(t)\Phi^{J}_{h}CKf\|^{p}dt \\
   & +c_{p}\int_{0}^{\tau}\|\Phi^{J}_{t}C[(Kf)(\cdot+h)-Kf]\|^{p}dt.
  \end{align*}
Let $M\ge 1$ and $\tilde{\om}\in\R$ such that $\|T^{cl}(t)\|\le M e^{\tilde{\om} t}$ for any $t\ge 0$ and $p>1$ such that $\frac{1}{p}+\frac{1}{q}=1$. By using \eqref{phi-comparaison},
\begin{align*}
\int_{0}^{\tau}\|T^{cl}(t)\Phi^{J}_{h}C(Kf)(\cdot)\|^{p}dt& \le M \t e^{|\tilde{\om}|\t} \|\Phi^{J}_{h}\|^p \|C Kf\|^{p}_{L^p([0,h],U)} \cr & \le  M \t e^{|\tilde{\om}|\t} \|\Phi^{J}_{h}\|^p \beta h^{\frac{p}{q}} \|f\|_p
\cr & \le  M \t e^{|\tilde{\om}|\t} \|\Phi^{J}_{\t}\|^p \beta h^{\frac{p}{q}} \|f\|_p \underset{h\to 0}{\longrightarrow} 0
\end{align*}
uniformly in $f$ such that $\|f\|_p\le 1$, where $\beta$ is a constant independent of $f$. On the other hand, by using \eqref{phi-comparaison},
\begin{align*}
\int_{0}^{\tau}\|\Phi^{J}_{t}C[(Kf)(\cdot+h)-Kf]\|^{p}dt\le \|\Phi^{J}_{\t}\| \|C\|^p \int_{0}^{\tau}\|(Kf)(t+h)-(Kf)(t)\|^{p}dt,
\end{align*}
which goes to $0$ as $h\to 0$ uniformly in $\|f\|_p\le 1$, due to the norm continuity of $\T$ and Theorem \ref{L.Weis}, respectively. This shows that
\begin{align*}
\int_{0}^{\tau}\|(K^{cl}f)(t+h)-(K^{cl}f)(t)\|^{p}dt \underset{h\to 0}{\longrightarrow} 0
\end{align*}
 uniformly in $\|f\|_p\le 1$. Now according to Theorem  \ref{L.Weis},  $\T^{cl}$ is immediately norm continuous.
\end{proof}

The following result on immediate norm continuity for perturbed semigroups can be considered as a generalization of a result proved in \cite{Matrai}.
\begin{thm}\label{main thm I N C}
Let assumptions of Theorem \ref{staffans-weiss perturbation} be satisfied with $C$ is a zero-class admissible. In addition we suppose that the semigroup $\T=(T(t))_{t\ge 0}$ is immediately norm continuous on $X$. Then the perturbed semigroup $\T^{cl}=(T^{cl}(t))_{t\ge 0}$ is immediately norm continuous on $X$ as well.
\end{thm}
\begin{proof}
By assumption, for any $t>0$ we have
\begin{align}\label{T-NC}
\lim_{h\to 0}\|T(t+h)-T(t)\|=0.
\end{align}
Now let us prove that the perturbed semigroup $\T^{cl}$ introduced in Theorem \ref{staffans-weiss perturbation} have also the above property. Due to \eqref{convolution R}, we have
\begin{align}\label{perturbed semigroup}
T^{cl}(t)x=T(t)x+R(t)x
\end{align}
for all $t\ge 0$ and $x\in X,$ where
\begin{align*}
R(t)x=\int_{0}^{t}T^{cl}_{-1}(t-s)JBC_{\Lambda}T(s)xds.
\end{align*}
According to \eqref{T-NC}, it suffices to prove the map $t\in (0,\infty)\mapsto R(t)$ is norm continuous. In fact, fix $t_0>0$ and  choose arbitrary $h,\delta\in\mathbb{R}$  such that $0<|h|<\delta<t$. We then have
\begin{align}\label{R-R}
\begin{split}
R(t+h)x-&R(t)x=\int_{0}^{\delta+h}T^{cl}_{-1}(t+h-s)JBC_{\Lambda}T(s)xds
+\int_{\delta+h}^{t+h}T^{cl}_{-1}(t+h-s)JBC_{\Lambda}T(s)xds\\
& -\int_{0}^{\delta}T^{cl}_{-1}(t-s)JBC_{\Lambda}T(s)xds-\int_{\delta}^{t}T^{cl}_{-1}(t-s)JBC_{\Lambda}T(s)xds\\
& =T^{cl}(t-\delta)\left[\int_{0}^{\delta+h}T^{cl}_{-1}(\delta+h-s)JBC_{\Lambda}T(s)xds - \int_{0}^{\delta}T^{cl}_{-1}(\delta-s)JBC_{\Lambda}T(s)xds \right]\\
&+\int_{\delta}^{t}T^{cl}_{-1}(t-s)JBC_{\Lambda}\left[T(s+h)x-T(s)x\right]ds\\
&:=I_{1}(x,h,t)+I_{2}(x,h,t),
\end{split}
\end{align}
where
\begin{align*}
I_{1}(x,h,t)&:=T^{cl}(t-\delta)\left[\Phi^J_{\delta+h}C_\Lambda T(\cdot)x-\Phi^J_{\delta}C_\Lambda T(\cdot)x \right],\\
I_{2}(x,h,t)&:=\int_{\delta}^{t}T^{cl}_{-1}(t-s)JBC_{\Lambda}\left[T(s+h)x-T(s)x\right]ds
\end{align*}
and
\begin{align*}
\Phi^J_\tau v=\int^\tau_0 T^{cl}_{-1}(\tau-s)JB v(s)ds,\qquad\tau\ge 0,\; v\in L^p([0,+\infty),U).
\end{align*}
By admissibility of $JB$ for $T^{cl}$, we have
\begin{align*}
\|\Phi^J_\tau v\|\le \|\Phi^J_\tau \|\|v\|_{L^p([0,\t],U)}
\end{align*}
for any $\t\ge 0$ and $v\in L^p([0,+\infty),U)$. Let $\tilde{M}\ge 1$ and $\om\in\R$ such that  $\|T^{cl}(t)\|\leq \tilde{M} e^{\omega t}$ for any $t\ge 0$. We estimate
\begin{align}\label{Tata}
\|I_{1}(x,h,t)\|&=\left\|T^{cl}(t-\delta)\left(\Phi^J_{\delta+h}C_\Lambda T(\cdot)x-\Phi^J_{\delta}C_\Lambda T(\cdot)x\right)\right\|\cr & \leq \tilde{M} e^{|\omega| (t-\delta)} \left(\|\Phi^J_{\delta+h}\|\|C_\Lambda T(\cdot)x\|_{L^p([0,\delta+h],U)}+ \|\Phi^J_{\delta}\|\|C_\Lambda T(\cdot)x\|_{L^p([0,\delta],U)}\right)\cr & \leq 2 \tilde{M} e^{|\omega| t} \|\Phi^J_{2\delta}\|\|C_\Lambda T(\cdot)x\|_{L^p([0,2\delta],U)}\cr & \le
2 \tilde{M} \ga(2\delta) e^{|\omega| t} \|\Phi^J_{2\delta}\|  \|x\|,
\end{align}
where $\ga(2\delta)\to 0$ as $\delta\to 0$, by \eqref{phi-comparaison} and  admissibility of $C$ for $T$. We then have
\begin{align}\label{I-1-estimate}
\|I_{1}(x,h,t)\|\le \varpi_t(\delta) \|x\|
\end{align}
for any $x\in X$, where $\varpi_t(\delta):=2 \tilde{M} \ga(2\delta) e^{|\omega| t} \|\Phi^J_{2\delta}\| \to 0$ as $\delta\to 0$. By admissibility of $JB$ for $\T^{cl}$, a change of variables and admissibility of $C$ for $\T$, we obtain
\begin{align}\label{I-2-estimate}
\begin{split}
\|I_{2}(x,h,t)\|&
=\left\|\int_{0}^{t}T^{cl}_{-1}(t-s)JBC_{\Lambda}\left[T(s+h)x-T(s)x\right]\mathbb{I}_{[\delta,t]}(s)ds\right\| \\
 & \leq \|\Phi^J_t\|\left(\int_{\delta}^{t}\|C_{\Lambda}T(s-\delta)\left[T(\delta+h)x-T(\delta)x\right]\|^p ds\right)^{1/p}\\
 & \leq \|\Phi^J_t\|\left(\int_{0}^{t-\delta}\|C_{\Lambda}T(s)[T(\delta+h)x-T(\delta)x]\|^p \right)^{1/p}ds\\
 & \leq \|\Phi^J_t\|\left(\int_{0}^{t}\|C_{\Lambda}T(s)[T(\delta+h)x-T(\delta)x]\|^p ds\right)^{1/p}\\
 &  \leq \gamma(t) \|\Phi^J_t\|\|T(\delta+h)x-T(\delta)x\|.
\end{split}
\end{align}
Combining \eqref{R-R}, \eqref{I-1-estimate} and \eqref{I-2-estimate}, we have
\begin{align*}
\|R(t+h)-R(t)\|\leq \varpi_t(\delta)+\gamma(t)\|\Phi^J_t\|\|T(\delta+h)-T(\delta)\|.
\end{align*}
The fact that the semigroup $\T$ is immediately  norm continuous implies that
\begin{align*}
\lim_{h\to 0}\|R(t+h)-R(t)\|\le \varpi_t(\delta).
\end{align*}
By letting $\delta\to 0,$ we obtain
\begin{align*}
\lim_{h\to 0}\|R(t+h)-R(t)\|=0.
\end{align*}
This ends the proof.
\end{proof}

The next result is about immediate compactness of perturbed semigroups.
\begin{thm}\label{main thm I C}
Let assumptions of Theorem \ref{staffans-weiss perturbation} be satisfied with $C$ is a zero-class admissible. In addition we suppose that the semigroup $\T=(T(t))_{t\ge 0}$ is immediately compact on $X$. Then perturbed semigroup $\T^{cl}=(T^{cl}(t))_{t\ge 0}$ is immediately compact on $X$ as well.
\end{thm}
\begin{proof}
According to \eqref{perturbed semigroup} and the immediate compactness of the semigroup $\T,$ it suffices to prove that the operators $R(t)$ are compact for any $t>0$.  For this, we shall use an approximation argument. Take $\epsilon>0$ and define
\begin{align*}
 R_{\epsilon}(t)x=\int_{\epsilon}^{t}T^{cl}_{-1}(t-s)JBC_{\Lambda}T(s)xds,\qquad t>\varepsilon.
\end{align*}
We now show that $R_{\epsilon}(t)$ approaches to $R(t)$ uniformly as $\epsilon\to 0$. By admissibility of $JB$ for $\T^{cl}$ and $C$ for $\T$, we obtain
\begin{align}\label{Comp-esti}
  \|R(t)x-R_{\epsilon}(t)x\| & =\left\|T^{cl}(t-\epsilon)\int_{0}^{\epsilon}T^{cl}_{-1}(\epsilon-s)JBC_{\Lambda}T(s)xds\right\|\\
  & \leq \|T^{cl}(t-\epsilon)\|\|\Phi^J_{\epsilon}\|\left[\int_{0}^{\epsilon}\|C_{\Lambda}T(s)x\|^pds\right]^{\frac{1}{p}}\\
  & \leq \gamma(\epsilon) \|T^{cl}(t-\epsilon)\|\|\Phi^J_{\epsilon}\|\|x\|.
\end{align}
This shows that $\|R(t)-R_{\epsilon}(t)\|\longrightarrow 0$ as $\epsilon \to 0$. Thus it suffices to show that  $R_{\epsilon}(t)$ is compact for $t>\epsilon$. Let us consider a sequence $(x_n)\subset X$ with $\|x_n\|\leq 1$. Since $T(\epsilon)$ is compact, then there exists a subsequence $x_{\varphi(n)}$ such that
\begin{align*}
  T(\epsilon)x_{\varphi(n)}\longrightarrow y\in X \quad as \quad n\to\infty.
\end{align*}
Hence
\begin{align*}
 R_{\epsilon}(t)x_{\varphi(n)}&=\int_{\epsilon}^{t}T^{cl}_{-1}(t-s)JBC_{\Lambda}T(s-\epsilon)T(\epsilon)x_{\varphi(n)}ds\\
 &=\int_{\epsilon}^{t}T^{cl}_{-1}(t-s)JBC_{\Lambda}T(s-\epsilon)\left[T(\epsilon)x_{\varphi(n)}-y\right]ds\\
 &\quad \quad +\int_{0}^{t-\epsilon}T^{cl}_{-1}(t-\epsilon-s)JBC_{\Lambda}T(s)yds\\
 &:=K_{\epsilon,t}(T(\epsilon)x_{\varphi(n)}-y)+R(t-\epsilon)y.
\end{align*}
Where
\begin{equation*}
  K_{\epsilon,t}x:=\int_{\epsilon}^{t}T^{cl}_{-1}(t-s)JBC_{\Lambda}T(s-\epsilon)xds.
\end{equation*}
On the other hand,
\begin{align*}
  \|K_{\epsilon,t}(T(\epsilon)x_{\varphi(n)}-y)\| & =\left\|\int_{0}^{t}T^{cl}_{-1}(t-s)JBC_{\Lambda}T(s-\epsilon)
  \left[T(\epsilon)x_{\varphi(n)}-y\right]\mathbb{I}_{[\epsilon,t]}(s)ds \right\| \\
   & \leq \|\Phi^J_t\|\left[\int_{\epsilon}^{t}\|C_{\Lambda}T(s-\epsilon)\left[T(\epsilon)x_{\varphi(n)}-y\right]\|^pds\right]^{\frac{1}{p}} \\
   & \leq \|\Phi^J_t\|\left[\int_{0}^{t}\|C_{\Lambda}T(s)\left[T(\epsilon)x_{\varphi(n)}-y\right]\|^pds\right]^{\frac{1}{p}}\\
   & \leq \gamma(t) \|\Phi^J_t\|\|T(\epsilon)x_{\varphi(n)}-y\|\longrightarrow 0 \quad as \quad n\rightarrow\infty.
\end{align*}
Whence, $R_{\epsilon}(t)x_{\varphi(n)}$ converge to $R(t-\epsilon)y\in X$ which means that $R_{\epsilon}(t)$ are compact for $t>\epsilon$. This ends the proof.
\end{proof}
\begin{rem}\label{NC-Miyadera} As  consequences of  Theorem \ref{main thm I N C} and Theorem \ref{main thm I C}, we have
\begin{enumerate}
  \item If $(A,B)$ is well-posed (with $B\in \calL(U,X_{-1})$) and $C$ is bounded (i.e. $C\in\calL(X,U)$), then assumptions of Theorem \ref{staffans-weiss perturbation} are satisfied and $C $ is a zero class observation operator. Then  $\T^{cl}$ is immediately norm continuous (resp. compact) whenever the semigroup $\T$ is.
  \item If $(A,C)$ is well-posed (with $C\in\calL(D(A),U)$) and $B$ is bounded (i.e. $B\in \calL(U,X)$), then assumptions of Theorem \ref{staffans-weiss perturbation} are satisfied. Let $\Phi^J_\t$ as in the proof of Theorem \ref{main thm I N C}. Using the boundedness of $B$ and H\"{o}lder inequality, we obtain
      \begin{align*}
      \left\|\Phi^J_\t\right\|\le \|B\| \t^{\frac{1}{q}}
      \end{align*}
      with $\frac{1}{p}+\frac{1}{q}=1$. Then by \eqref{Tata} we have $\|I_1(x,h,t)\|\to 0$ as $h\to 0$ uniformly in $x$. Then without assuming zero class property for $C,$ we obtain that $\T^{cl}$ is immediately norm continuous (resp. compact) whenever the semigroup $\T$ is.
\end{enumerate}
\end{rem}
The following theorem is a generalisation of a result  proved  in \cite{LiGuHu} in  the case of Miyadera-Voigt perturbations.
\begin{thm}\label{boulouz}
  Let assumptions of Theorem \ref{staffans-weiss perturbation} be satisfied with $C$ is a zero-class admissible. Then $R(t):=T^{cl}(t)-T(t)$ is compact for $t>0$ if and only if $R(t)$ is norm continuous for $t\geq 0$ and $R(\lambda,A^{cl})-R(\lambda,A)$ is compact for $\lambda\in\rho(A^{cl})\cap \rho(A)$.
\end{thm}
\begin{proof}
  The sufficient condition  can be obtained by the same arguments as in \cite{LiGuHu}.  Let us now prove the necessary conditions. The  compactness of $R(\lambda,A^{cl})-R(\lambda,A)$ is obtained   by taking  Laplace transform of compact operators $R(t)$ and using the result \cite[theorem 3.3]{Pazy}. On the other hand for $t>0$ and $h$ near of zero, we have
  \begin{align*}
    R(t+h)-R(t) & =T^{cl}(t+h)-T(t+h)-R(t) \\
     & =T^{cl}(h)(T^{cl}(t)-T(t))+(T^{cl}(h)-T(h))T(t)-R(t)\\
     & =(T^{cl}(h)-I)R(t)+R(h)T(t).
  \end{align*}
 Now the compactness of  $R(t)$ implies that
  $$\|(T^{cl}(h)-I)R(t)\|\longrightarrow 0 \quad \mbox{ as } h\to 0.$$
  Moreover, using admissibility of $JB$ for $\T^{cl}$ and admissibility of $C$ for semigroup $\T$, we obtain
  \begin{align*}
    \|R(h)x\| & =\left\|\displaystyle\int_{0}^{h}T^{cl}_{-1}(h-s)JBC_{\Lambda}T(s)xds\right\| \\
     &\leq  \|\Phi_h^J\|\left(\displaystyle\int_{0}^{h}\|C_{\Lambda}T(t)x\|^pds\right)^{1/p}\\
     &\leq \gamma(h) \|\Phi_h^J\|\|x\|
  \end{align*}
  converges to 0 as $h\to 0$ for every $x\in X$. Hence,
  \begin{align*}
    \|R(h)\|\longrightarrow 0 \quad \mbox{ as } h\to 0.
  \end{align*}
  This ends the proof.
\end{proof}
\begin{rem}
In Theorems \ref{main thm I N C},\ref{main thm I C} and  \ref{boulouz}, we can replace the condition $C$ is zero class observation operator by a similar dual concept on admissible control operator $B$. In fact, let $B\in\calL(U,X_{-1})$ be an admissible control operator for $A$ with control maps $\Phi_t,\;t\ge 0$. For any $\t>0,$ there exists $c(\t)>0$ such that
\begin{align}\label{Estim-phi}
\left\|\Phi_\t u\right\|\le c(\t) \|u\|_p
\end{align}
for all $u\in L^p(\R^+,U)$. Now we say that $B$ is a zero class control operator if the constant $c(\t)\to 0$ as $\t\to 0$. This notion is used in \cite{JNPS} to study input-to-state stability for the infinite-dimensional systems. Let assumptions of Theorem \ref{staffans-weiss perturbation} be satisfied. From \eqref{isomorphism J}, we have
\begin{align*}
\Phi^J_\t C_\Lambda T(\cdot)x= \Phi_\t  C_\Lambda T^{cl}(\cdot)x
\end{align*}
for any $\t\ge 0$ and $x\in X$. According to \eqref{Estim-phi} and the admissibility of $C_\La$ for $\T^{cl},$ there exists a constant $\tilde{\ga}>0$ such that
\begin{align}\label{llll}
\left\|\Phi^J_\t C_\Lambda T(\cdot)x\right\|\le c(\t)\tilde{\ga} \|x\|.
\end{align}
Thus if in  Theorems \ref{main thm I N C},\ref{main thm I C} and  \ref{boulouz}  instead of $C$ is a zero class observation operator  we assume that  $B$ is a zero class control operator, then we obtain the same results. In fact, in the proof of these theorems we replace the fact that $\ga(\t)\to 0$ as $\t\to 0$ by $c(\t)\to 0$ as $\to 0$, due to \eqref{Tata}, \eqref{Comp-esti}, \eqref{Estim-phi}and \eqref{llll}.
\end{rem}
\begin{exa}
Consider a one dimensional heat equation with mixed boundary conditions
\begin{align}\label{heat equation}
\left\{
  \begin{array}{ll}
\vspace{0,2cm}
    \displaystyle \frac{\partial}{\partial t}z(t,x)=\frac{\partial^2}{\partial x^2}z(t,x), & \hbox{$0<x<\pi$, $t\geq0$;} \\
\vspace{0,2cm}
    \displaystyle\frac{\partial}{\partial x}z(t,0)+z(t,0)=0,\quad z(t,\pi)=0, & \hbox{$t\geq0$;} \\
    z(0,x)=\varphi(x), & \hbox{$0<x<\pi$.}
  \end{array}
\right.
\end{align}
In order to use our abstract results, we select
\begin{align*}
X:=L^2([0,\pi]), \quad Z:=\left\{f\in H^2([0,\pi]):f(\pi)=0\right\}, \quad \partial X:=\mathbb{C}
\end{align*}
and operators
\begin{align*}
A_mf=f'', \quad Gf=f'(0)\quad and \quad Mf=-f(0), \quad for \quad f\in Z.
\end{align*}
We know that  the operator
\begin{align*}
  A\varphi:=A_m \varphi,\quad D(A)=\{f\in Z: f'(0)=0\}
\end{align*}
generates an immediately norm continuous (even compact) $C_0$-semigroup $\T:=(T(t))_{t\geq 0}$ on $X$ (note that $A$ is self-adjoint). On the other hand, $G$ is surjective. So that the Dirichlet operator $D_\la$ exists for any $\la\in\rho(A)$, see the beginning of this section. As we are in the Hilbert setting, the extrapolation space $X_{-1}$ of $X$ associated with $A$ is isomorph to the topological dual space $(D(A^\ast))',$ where $A^\ast$ is the adjoint operator of $A$. We now put for any $\la\in\rho(A),$
\begin{align*}
B:=(\lambda-A_{-1})D_{\lambda}\in \mathcal{L}(\mathbb{C},D(A^{\ast})^{\prime}).
\end{align*}
A straightforward computation shows  that the adjoint operator of $B$ is given by
\begin{align*}
B^{\ast}\varphi=-\varphi(0),\quad \varphi\in D(A^{\ast})=D(A).
\end{align*}
In addition, $B^{\ast}$ is an admissible observation operator for the adjoint semigroup $\T^{\ast}:=(T^{\ast}(t))_{t\geq 0}$. Hence, by duality, $B$ is an admissible control operator for  the semigroup $\T$, (see \cite[p.126]{Tucsnak Weiss}). In addition, $B^{\ast}$ is a zero-class observation operator  (see \cite[example 3.8]{Jacob Partington Pott}).\\
 Moreover, by computation the Dirichlet operator is given by
\begin{align*}
  (D_0 u)(x) & =(x-\pi)\cdot u ,\quad for \quad 0\leq x\leq \pi;\\
  (D_{\lambda} u)(x) & = \frac{\sinh(\sqrt{\lambda}(x-\pi))}{\sqrt{\lambda}\cosh(\sqrt{\lambda}\pi)}\cdot u,\quad for \quad 0\leq x\leq \pi\quad and \quad \lambda>0.
\end{align*}
It is clear that $Range \left(D_{\lambda}\right)\subset Z$ and the transfer function $H(\lambda):=M D_{\lambda}$ is uniformly bounded on half plan. Consequently, $(A,B,B^{\ast})$ generates a regular system on $L^2([0,\pi]),\C,\C$. It is well know that $A$ generates a compact semigroup in $X$. Hence by  Theorem \ref{main thm I C} the semigroup solution of heat equation with mixed boundary \eqref{heat equation} is compact in $X$.
\end{exa}

\section{The Eventual differentiability under Desch-Schappacher perturbations}\label{eventual-section}
In this section, we still assume that assumptions of Theorem  \ref{staffans-weiss perturbation} are satisfied (hence the perturbed semigroup $\T^{cl}$ exists), and then  discuss conditions for which there exists $\t>0$ such that for any $x\in X$ the map $t\in (\t,\infty)\to T^{cl}(t)x\in X$ is differentiable.

 The following result due to Pazy \cite[p.57]{Pazy} (see also \cite{Pazy1}) gives  a sufficient condition for differentiability of $C_0$--semigrpups on Banach spaces.
\begin{thm}\label{Pazy}
  Let $G:D(G)\subset X\to X$ be the generator of a $C_{0}$-semigroup $V:=(V(t))_{t\geq 0}$  on $X$.   If for some $\mu>\om_0(G)$ we have
  \begin{equation}\label{Pazy condition}
    \tau_0:=\limsup_{|\tau|\rightarrow +\infty}\log {|\tau|}\|R(\mu+i\tau,G)\|<\infty .
  \end{equation}
  Then $V(t)$ is differentiable for $t>\tau_0$. In addition, the semigroup  $V$ is immediately differentiable  whenever $\tau_0=0$.
\end{thm}
In the following result we  generalize a Pazy result on the stability of differentiability under bounded perturbation (see \cite{Pazy}) to Desch-Schappacher  perturbations operators.
\begin{thm}\label{main thm differetiable}
Let the control system $(A,B)$ be well-posed on $X,U$  and $C\in\mathcal{L}(X)$. Assume that the generator  $(A,D(A))$ satisfies the condition  \eqref{Pazy condition}. Then the assumption of Theorem \ref{staffans-weiss perturbation} are satisfied and the perturbed semigroup $\T^{cl}$ is eventually differential from a time $2\tau_0$. On the other hand, if $\t_0=0$ then the semigroup $\T^{cl}$ is immediately differentiable.
\end{thm}
\begin{proof} According to Remark \ref{DS-MV-perturbations}, the part of $A_{-1}+BC$ on $X$ denoted by $A^{cl}$ generates a strongly continuous semigroup $\T^{cl}$ given by the variation of constants formula \eqref{VCF-T-cl-desch}. Taking Laplace transform on both sides of this formula we obtain
\begin{align*}
R(\lambda,A^{cl})=R(\lambda,A)+R(\lambda,A_{-1})BCR(\lambda,A^{cl}),\qquad \la\in \rho(A).
\end{align*}
Let $\om>\om_0(A)$. According to \eqref{adm cont weiss} there exists a contants $c>0$ such that for any $\la\in\C$ with ${\rm Re}\la>\om,$
\begin{align*}
\|R(\lambda,A^{cl})\|=\|R(\lambda,A)\|+ \frac{c\|C\|}{\left({\rm Re}\la-\om\right)^{\frac{1}{p}}} \|R(\lambda,A^{cl})\|
\end{align*}
Now, for $\mu>\om+(2 c \|C\|)^p:=\om',$ and $\t\in\R$, we obtain
\begin{align}\label{R-cl-R}
\|R(\mu+i\tau,A^{cl})\|\le 2 \|R(\mu+i\tau,A)\|.
\end{align}
Thus the result follows from Theorem \ref{Pazy}.
\end{proof}
\begin{rem}\label{Pazy-condition-MV}
 Assume that $A:D(A)\subset X\to X$ is a generator of a $C_0$--semigroup $\T$ on $X$, $B\in\calL(X)$ and $C\in\calL(D(A),X)$ such that $(C,A)$ is well-posed. Then the condition of Theorem \ref{staffans-weiss perturbation} are verified and the operator $A^{cl}=A+BC$ with domain $D(A^{cl})=D(A)$, generates a $C_0$-semigroup $\T^{cl}$ on $X$ (see also \cite{Hadd-SF}). Moreover,
 \begin{align*}
 R(\la,A^{cl})=R(\la,A)+R(\la,A^{cl})BCR(\la,A),\qquad \la\in\rho(A).
 \end{align*}
Let $\mu\in\R$ such that $\mu>\om>\om_0$ and $\t\in\R$. According to \eqref{adm obs weiss}, there exists a constant $c>0$ such that
\begin{align*}
\|R(\mu+i\tau,A^{cl})\|\le \|R(\mu+i\tau,A)\|+ \frac{c \|B\|}{\left(\mu-\om\right)^{1-\frac{1}{p}}} \|R(\mu+i\tau,A^{cl})\|.
\end{align*}
Now for $\mu>\om+(2c\|B\|)^{\frac{p}{p-1}},$ we have
\begin{align*}
\|R(\mu+i\tau,A^{cl})\|\le 2 \|R(\mu+i\tau,A)\|.
\end{align*}
As $A$ satisfies the condition \eqref{Pazy condition},   by Theorem \ref{Pazy}, the generator of the perturbed semigroup   $\T^{cl}$ satisfies also this condition and hence $\T^{cl}$ is a differential semigroup.
\end{rem}
We also have the following observation about immediate norm continuity for perturbed semigroups on Hilbert spaces.
\begin{rem}\label{Hilbert-imm-norm-cont}
Assume that we work in the Hilbert setting and let us in the situations of Theorem \ref{main thm differetiable} and/or Remark \ref{Pazy-condition-MV}. In both cases we have proved that the estimate \eqref{R-cl-R} for the generators $A$ and $A^{cl}$. It is well-known (see e.g. \cite[p.115]{Engel Nagel}) that $\T$ is immediately norm continuous on a Hilbert space $X$ if and only if $\|R(\mu+i\t,A)\|\to 0$ as $\t\to\pm\infty$. Now the inequality  \eqref{R-cl-R} shows that the immediate norm continuity is stable under Miyadera-Voigt and/or Desch-Schappacher perturbations.
\end{rem}

\section{Application to boundary integro-differential equations}\label{Appl-Volterra}
let $X$ and $Z$ be  Banach spaces with $Z\hookrightarrow X$ continuous and dense embedding. Let $A_m:Z\to X$ be a closed linear (differential) operator and an application $k:\R^+\to \C$ a measurable function. Consider the following boundary integro-differential equation
 \begin{equation}\label{Voltera 2}
  \left\{
    \begin{array}{ll}
      \dot{x}(t)=A_{m}x(t)+\displaystyle\int_{0}^{t}k(t-s)Px(s)ds, & \hbox{$t\geq 0$} \\
       Gx(t)=Mx(t)  , & \hbox{$t\geq 0$,}\\
      x(0)=x, & \hbox{}
    \end{array}
  \right.
\end{equation}
where the initial condition $x\in X$ and boundary operators $G:Z\longrightarrow \partial X$ and $M: Z\longrightarrow \partial X$ are linear.

The objective of this section is to study the well-posedness of the equation \eqref{Voltera 2} and establish regularity of the solution.
In the spirit of Greiner \cite{Greiner} and Salamon \cite{Salamon}, we introduce the hypothesis
\begin{itemize}
  \item[{\bf(H1)}] $A:=A_m$ with domain $D(A)=\ker G$ generates a $C_{0}$-semigroup $(T(t))_{t\geq 0}$ on $X$.
  \item[{\bf(H2)}] The operator $G:Z\longrightarrow \partial X$ is surjective.
\end{itemize}
As discussed in the introductory section, let $D_\la$ be the Dirichlet operator associated with $A_m$ and $G$ and set $B:=(\lambda - A_{-1})D_{\lambda}$ for $\la\in \rho(A)$. On the other hand, we select
\begin{align}\label{restrection-operators}
C:=M_{|D(A)}, \quad\text{and}\quad  \P:={P}_{|D(A)}.
\end{align}
We further assume that
\begin{itemize}
  \item[{\bf (H3)}] $(A,B,C)$ is a regular system on $X,\partial X, \partial X$ and $I:\partial X \longrightarrow \partial X$ is a feedback admissible with $C$ is a zero-class admissible.
  \item[{\bf (H4)}] $(A,B,\P)$ is a regular system on $X,\partial X, X$.
\end{itemize}
In order to study the existence and regularity  of the solution of the integro-differential equation \eqref{Voltera 2}, we need to introduce a Bergman space. Let then $h:\mathbb{R}^{+}\longrightarrow \mathbb{R}^{+}$ be an admissible function (i.e $h$ is increasing, convex and $h(0)=0$). Hereafter, we assume that for $s>1,$
\begin{align*}
\int^1_0 h(\si)^{1-s}d\si<\infty.
\end{align*}
Let $p,q\in (1,+\infty)$ be such that
\begin{align*}
q=\frac{ps}{s-1}.
\end{align*}
We define the sector
\begin{align*}
\Sigma_{h}:=\left\{\si+i\tau\in \mathbb{C}, \si>0\text{ and } |\tau|<h(\si)\right\}
\end{align*}
The Bergman space is defined by
\begin{align*}
\B^{q}_h(\Sigma_{h},X):= \left\{f:\Si_h\to X\;\text{holomorphic such that} \|f\|_{\B^{q}_h(\Sigma_{h},X)}<\infty\right\}
\end{align*}
with the norm
\begin{align*}
\|f\|_{\B^{q}_h(\Sigma_{h},X)}:=\left(\int\int_{\Sigma_{h}}\|f(\si+i\tau)\|^{q}d\si d\tau\right)^{\frac{1}{q}}<\infty.
\end{align*}
We shall assume
\begin{itemize}
  \item [{\bf (H5)}] $k(\cdot)\in \B^{q}_h(\Sigma_{h},\C)$.
\end{itemize}
According to \cite{Barta1},  the following  translation semigroup on $\B^{q}_h(\Sigma_{h},X),$
\begin{align*}
(S(t)f)(z):=f(t+z)
\end{align*}
with  generator
\begin{align*}
\frac{d}{dz}f=f', \quad D\left(\frac{d}{dz}\right):=\left\{f\in \B^{q}_h(\Sigma_{h},X), f'\in \B^{q}_h(\Sigma_{h},X)\right\}
\end{align*}
is analytic. We first use product spaces to reformulate the equation \eqref{Voltera 2} as abstract boundary value problem. In fact, consider the Banach space
\begin{align*}
\mathcal{X}:=X\times \B^{q}_h(\Sigma_{h},X)\quad \text{with norm}\quad \left\|(\begin{smallmatrix} x\\ f\end{smallmatrix})\right\|:=\|x\|+\|f\|_{\B^{q}_h(\Sigma_{h},X)}.
\end{align*}
Moreover we consider the space
\begin{align*}
\calZ:=Z\times D\left(\frac{d}{dz}\right).
\end{align*}
The equation  \eqref{Voltera 2} can be rewritten as
 \begin{equation}\label{Boundary Voltera}
  \left\{
    \begin{array}{ll}
      \dot{z}(t)=\mathcal{A}_{m}z(t)+\mathcal{P}z(t), & \hbox{$t\geq 0$} \\
       \mathcal {G}z(t)=\mathcal{M}z(t)  , & \hbox{$t\geq 0$,}\\
      z(0)=z\in \mathcal{X} & \hbox{.}
    \end{array}
  \right.
\end{equation}
where $\calA_m,\calP: \calZ\to \calX$ are given by
\begin{align*}
\calA_m:=\begin{pmatrix}
           A_m & 0 \\
           0 & \frac{d}{dz}
         \end{pmatrix}, \qquad \calP:=\begin{pmatrix}
           0 & \delta_0 \\
           k(\cdot)P & 0
         \end{pmatrix},
\end{align*}
and the boundary operators $\calG,\calM:\calZ\to \partial X$ are defined by
\begin{align*}
\calG:=\begin{pmatrix} G& 0 \end{pmatrix},\qquad \calM:=\begin{pmatrix} M& 0 \end{pmatrix}.
\end{align*}
Now consider the operator
\begin{align*}
\calA:=\calA_m,\quad D(\calA)=\left\{x\in Z:Gx=Mx\right\}\times D\left(\frac{d}{dz}\right).
\end{align*}
The boundary problem \eqref{Boundary Voltera} is reformulated again as a Cauchy problem
\begin{align}\label{Cauchy-Volterra}
\begin{cases}\dot{z}(t)=\calA z(t)+\calP z(t),& t\ge 0,\cr z(0)=z.  \end{cases}
\end{align}
\begin{lem}\label{generation-calA}
Let assumptions (H1),(H2) and (H3) be satisfied. The the operator $(\calA,D(\calA))$ is the generator of a strongly continuous semigroup $(\calT(t))_{t\ge 0}$ on $\calX,$ given by
\begin{align*}
\calT(t)=\begin{pmatrix}T^{cl}(t)& 0\\ 0& S(t)\end{pmatrix},\qquad t\ge 0,
\end{align*}
where $(T^{cl}(t))_{t\ge 0}$  is the strongly continuous semigroup on $X$ generated by the operator
\begin{align}\label{abc}
A^{cl}:=A_{-1}+BC_\La\quad\text{with}\quad D(A^{cl})=\left\{x\in D(C_\Lambda): (A_{-1}+BC_\La)x\in X\right\}.
\end{align}
\end{lem}
\begin{proof}
According to Remark \ref{unbounded perturbation domain}, the following operator
\begin{align*}
A^{cl}:=A_m,\quad D(A^{cl})=\{x\in Z:Gx=Mx\}
\end{align*}
coincides with the operator defined by \eqref{abc}, which is a generator of a strongly continuous semigroup $\T^{cl}:=(T^{cl}(t))_{t\ge 0}$, due to (H3) and Theorem \ref{staffans-weiss perturbation}. With these we have
\begin{align*}
\calA=\begin{pmatrix} A^{cl}& 0\\0& \frac{d}{dz}\end{pmatrix},\qquad D(\calA)=D(A^{cl})\times D\left(\frac{d}{dz}\right).
\end{align*}
This ends the proof.
\end{proof}
\begin{thm}\label{Well-posed-Volterra}
Let assumptions (H1) to (H5) be satisfied. Then the operator $(\calA+\calP, D(\calA))$ is a generator of a strongly continuous semigroup on $\calX$.
\end{thm}
\begin{proof}
According to \cite{Hadd-SF}, it suffices to prove that $\calP$ is an admissible operator for $\calA$. Let $(\begin{smallmatrix}x\\f\end{smallmatrix})\in D(\calA)$. As $\calT(t)(\begin{smallmatrix}x\\f\end{smallmatrix})\in D(\calA),$ then
\begin{align*}
T^{cl}(t)x\in D(A^{cl})\quad\text{and}\quad S(t)f\in D\left(\frac{d}{dz}\right),
\end{align*}
for any $t\ge 0$. According to \eqref{inclusion-important} and \eqref{restrection-operators}, we have
\begin{align*}
\calP\calT(t)(\begin{smallmatrix}x\\f\end{smallmatrix})=\begin{pmatrix} f(t)\\ k(\cdot)\P_{\Lambda}T^{cl}(t)x\end{pmatrix},\qquad t\ge 0.
\end{align*}
Here $\P_\Lambda$ is the Yosida extension of $\P$ relatively to $A,$. We recall that from feedback theory and the condition (H4), the operator $\P_\Lambda$ is an admissible observation operator for $\T^{cl}$. Then for  constants $\la>0$ and $\,p>1,$ there exist constants $\ga>0$ and $c_p>0$ such that
\begin{align*}
\int^\al_0 \left\|\calP \calT(t)(\begin{smallmatrix}x\\f\end{smallmatrix})\right\|^pdt&\le c_p \left(\int^\al_0 \|f(t)\|^p_Xdt+\int^\al_0 \|k(\cdot)\P_\La T^{cl}(t)x\|^p_{\B^q_h(\Sigma_h,X)}dt\right)\cr & \le   c_p \int^\al_0 \|f(t)\|^p_Xdt+ \int^\al_0 \left(\int\int_{\Si_h} \|k(\si+i\tau)\P_\La T^{cl}(t)x\|^qd\si d\tau\right)^{\frac{p}{q}}dt\cr & \le c_p \int^\al_0 \|f(t)\|^p_Xdt+c_p \|k\|^p_{\B^q_h(\Sigma_h,\C)} \int^\al_0\|\P_\La T^{cl}(t)x\|^p dt\cr &\le c_p \int^\al_0 \|f(t)\|^p_Xdt+c_p \ga^p \|k\|^p_{\B^q_h(\Sigma_h,\C)}\|x\|^p.
\end{align*}
On the other hand, using Cauchy formula, Jensen's inequality and similar arguments as in \cite[Lem.4.3]{Barta2}, one can see that there exists a constant $\kappa>0$ such that
\begin{align*}
\int^\al_0 \|f(t)\|^p_Xdt\le \kappa \|f\|^p_{\B^q_h(\Si_h,X)}.
\end{align*}
Now by taking $\vartheta:=c_p\max\{\kappa, \ga^p \|k\|^p_{\B^q_h(\Sigma_h,\C)}\},$ we obtain
\begin{align*}
\int^\al_0 \left\|\calP \calT(t)(\begin{smallmatrix}x\\f\end{smallmatrix})\right\|^pdt\le \vartheta (\|x\|+\|f\|_{\B^q_h(\Si_h,X)})^p.
\end{align*}
This ends the proof.
\end{proof}
\begin{thm}\label{volterra-nc}
Let assumptions (H1) to (H5) be satisfied. Moreover, we assume that the operator $A$ generates an immediately norm continuous semigroup on $X$. Then the operator $(\calA+\calP,D(\calA))$ generates  an immediately norm continuous semigroup on $\calX$ as well.
\end{thm}
\begin{proof}
Theorem \ref{main thm I N C} show that the operator $A^{cl}$ generates an immediately norm continuous semigroup $\T^{cl}$ on $X$. On the other hand, according \cite{Barta1}, we know that the shift semigroup $S$ is analytic in the Bergman space $\B^q_h(\Si_h,X)$, hence it is immediately norm continuous semigroup. This show that the semigroup $(\calT(t))_{t\ge 0}$ generated by $\calA$ is immediately norm continuous. As $\calP$ is a Miyadera-Voigt perturbation for $\calA,$ then by Remark \ref{NC-Miyadera} the operator $(\calA+\calP,D(\calA))$ generates  an immediately norm continuous semigroup on $\calX$.
\end{proof}
\begin{thm}\label{volterra-ev-diff}
Assume that $M\in\calL(X,\partial X)$ and  conditions (H1) to (H5) be satisfied. If for some $\mu>\om_0(A)$ we have
\begin{align}\label{paz-cond}
\t_0:=\limsup_{|\t|\to\infty}\log(|\t|)\|R(\mu+i\t,A)\|<\infty,
\end{align}
then the semigroup generated by $(\calA+\calP,D(\calA))$ is eventually differentiable.
\end{thm}
\begin{proof}
As $A$ satisfies the Pazy condition \eqref{paz-cond}, then
by the proof of Theorem \ref{main thm differetiable}  the operator $A^{cl}$ satisfies also the Pazy condition. We know from \cite{Barta1} that the shift semigroup $(S(t))_{t\ge 0}$ is analytic on the Bergman space $\B^q_h(\Si_h,X)$. This implies that the operator $\calA$ satisfies the Pazy condition. Now as $(\calP,\calA)$ is admissible (see the proof of Theorem \ref{Well-posed-Volterra}), then by Remark \ref{Pazy-condition-MV}, the operator $(\calA+\calP,D(\calA))$ satisfies also the Pazy condition. It follows from Theorem \ref{Pazy} that the operator $(\calA+\calP,D(\calA))$ generates an eventually differentiable semigroup on $\calX$.
\end{proof}

\begin{rem}
In the case of $M\equiv 0$, the boundary integro-differential equation \eqref{Voltera 2} becomes
 \begin{equation}
  \left\{
    \begin{array}{ll}
      \dot{x}(t)=Ax(t)+\displaystyle\int_{0}^{t}k(t-s)Px(s)ds, & \hbox{$t\geq 0,$} \\
      x(0)=x\in X & \hbox{}
    \end{array}
  \right.
\end{equation}
In this case Bart\`{a} \cite{Barta2} showed the differentiability of the solutions by assuming a smooth regularity on the kernel  $k$ that is $k'\in \B^q_h(\Si_h,\C)$. However, in our results  this extra condition on $k$ is not needed any more. On the other hand, the approach of Bart\`{a} is based on small perturbations, and cannot be extended to Desch-Schappacher perturbations. We think that our approach based on feedback theory of regular linear systems is the right way to solve such problems.
\end{rem}

\end{document}